\documentclass[final,3p,times]{elsarticlearxiv}
\usepackage{graphicx}        % standard LaTeX graphics tool
\usepackage{multicol}        % used for the two-column index
\usepackage[bottom]{footmisc}% places footnotes at page bottom
\usepackage{mdframed}

\usepackage[english]{babel}
\usepackage{amssymb,amstext,amsmath,amsthm,mathabx}
\usepackage{hyperref}
\usepackage{array}   % for \newcolumntype macro
\newcolumntype{L}{>{$}l<{$}}
\theoremstyle{definition}
\newtheorem{definition}{Definition} 
\newtheorem{remark}{Remark} 
\newtheorem{theorem}{Theorem} 
 
\newcommand{\no}[1]{\widebar{#1}}
\renewcommand{\bot}{\emptyset}

\def\F{\mathcal F}
\def\P{\mathcal P}
\def\M{\mathcal M}
\def\I{\mathcal I}
\def\H{\mathcal H}
\def\pr{\mathbb{P}}
\def\prev{\mathbb{P}}
\def\K{\mathcal{K}}

\usepackage{amssymb}
\journal{}

\begin{document}

\begin{frontmatter}



%\title{Centering  with  
%	conjoined and iterated conditionals under coherence\footnote{This is a substantially extended version of a paper  (\cite{GOPS16}) presented at SMPS 2016 (Soft Methods in Probability and Statistics 2016) conference held in Rome in September 12--14, 2016.}}


%\author{Angelo Gilio\and
%David E. Over \and Niki Pfeifer  \and Giuseppe Sanfilippo  }
\title{Probabilistic inferences from conjoined to  iterated conditionals	\tnoteref{mytitlenote}
}
\tnotetext[mytitlenote]{
This is a substantially extended version of a paper   (\cite{GOPS16}) presented at the 8\textsuperscript{th} International Conference   Soft Methods in Probability and Statistics 2016 (SMPS 2016) held in Rome in September 12--14, 2016.
}

\author[gs]{Giuseppe Sanfilippo\corref{cor1}}
%\address[gs]{University of Palermo, Italy}
\address[gs]{Department of Mathematics and Computer Science, Via Archirafi 34, 90123 Palermo, Italy}
\ead{giuseppe.sanfilippo@unipa.it}
\author[np]{Niki Pfeifer}
%\address[np]{Ludwig-Maximilians-University Munich, Germany}
\address[np]{Munich Center for Mathematical Philosophy, Ludwig-Maximilians-University Munich, Ludwigstra\ss e 31, 80539, Munich, Germany}
\ead{niki.pfeifer@lmu.de}
\author[do]{David E.\ Over}
%\address[do]{University of Durham, UK}
%\address[do]{Department of Psychology, Durham University, Durham, United Kingdom}
\address[do]{Department of Psychology, Durham University, Science Site, South Road, Durham {DH1~3LE}, United Kingdom}
\ead{david.over@durham.ac.uk}
\author[ag]{Angelo Gilio\fnref{fn1}}
%\address[ag]{University of Rome ``La Sapienza'', Italy}
\address[ag]{Department of Basic and Applied Sciences for Engineering, University of Rome ``La Sapienza'', Via A. Scarpa 14, 00161 Roma, Italy}
\ead{angelo.gilio@sbai.uniroma1.it}
\fntext[fn1]{Retired}
\cortext[cor1]{Corresponding author}

\begin{abstract}
There is wide support in logic, philosophy, and psychology for the hypothesis that the probability of the indicative conditional of natural language,  $P(\textit{if } A \textit{ then } B)$, is the conditional probability of $B$ given $A$, $P(B|A)$. We identify a conditional which is such that $P(\textit{if } A \textit{ then } B)= P(B|A)$ with de Finetti's conditional event, $B|A$. An objection to making this identification in the past was that it appeared unclear how to form compounds and iterations of conditional events. In this paper, we illustrate how to overcome this objection with a probabilistic analysis, based on coherence, of these compounds and iterations. We interpret the compounds and iterations as conditional random quantities which, given some logical dependencies, may reduce to conditional events.
We  show  how the inference to $B|A$ from  $A$ and $B$  can be extended to compounds and iterations of both conditional events and biconditional events.
Moreover,  we determine the respective uncertainty propagation rules.
Finally, we make some comments on extending our analysis to counterfactuals.
\end{abstract}


\end{frontmatter}



\section{Introduction}
There is wide agreement in logic and philosophy that the indicative conditional of natural language, \textit{if A then
B}, cannot be adequately represented as the material conditional of binary  logic, logically equivalent to
$\no A \vee B$ (\textit{not-$A$ or $B$}) \cite{edgington14}. Psychological studies have also shown that ordinary people do not judge the
probability of $\textit{if } A \textit{ then } B$,  $P(\textit{if } A \textit{ then } B)$, to be the probability of the material conditional,
$P(\no A \vee B)$, but rather tend to assess it as the conditional probability of $B$ given $A$, $P(B|A)$, or at least to converge on this assessment \cite{baratgin13a,evans03,fugard11a,Takahashi17,pfeifer12x,pfeifer13b,singmann14}. These
psychological results have been taken to imply \cite{baratgin13a,evans04,gilio12,oaksford07,pfeifer13b,pfeifer13}, that $\textit{if } A \textit{ then } B$ is best represented, either as the probability
conditional of Adams \cite{adams98}, or as the \textit{conditional event} $B|A$ of de Finetti \cite{definetti36,definetti37},
the probability of which is $P(B | A)$. 
We will adopt the latter view in the present paper and base our analysis on conditional events and coherence (for related analyses, specifically on categorical syllogisms, squares of opposition under coherence and on generalized argument forms see \cite{gilio16,PfSa17,PS2017,ECSQARU17}).
 One possible
objection to holding that $P(\textit{if } A \textit{ then } B) = P(B | A)$ is that it is supposedly unclear how
this relation extends to compounds of conditionals and makes sense of them  \cite{douven15,edgington14,wijnbergen15}. Yet consider:
\begin{equation}\label{S:MotherPrem}
 \overbrace{\textit{She will be  angry}}^{a} \text{ if } \overbrace{\textit{her son gets a }\text{B}}^{b} \text{ and } \overbrace{\textit{she will be  furious}}^{f}   \text{ if } \overbrace{\textit{he gets a }\text{C}}^{c}.
\end{equation}
The above conjunction appears to make sense, as does the following seemingly even more complex conditional construction
\cite{douven15}:
\begin{equation}\label{S:MotherConcl}
\text{If } \textit{she will be angry} \text{ if } \textit{her son gets a }\text{B, then } \textit{she will be furious } \text{if } \textit{he gets a }\text{C}.
\end{equation}
We will show below, in reply to the objection, how to give sense to~(\ref{S:MotherPrem}) and~(\ref{S:MotherConcl}) 
in terms of \emph{compound conditionals}. Specifically, we will interpret~(\ref{S:MotherPrem}) as a  \emph{conjunction} of two conditionals ($a|b$ and $f|c$) and~(\ref{S:MotherConcl})  in terms of a \emph{conditional} whose antecedent ($a|b$) and consequent ($f|c$) are both conditionals (if $a|b$, then $f|c$). 
But we note first that the \emph{iterated conditional}~(\ref{S:MotherConcl}) validly follows from the conjunction~(\ref{S:MotherPrem}) by the form of inference we will call \textit{centering}  which, as we will show,  can be extended to  the compounds of conditionals (see Section~\ref{SEC:CENTERING} below). 
We point out  that our framework is  quantitative  rather than a logical one. 
Indeed in our approach, syntactically conjoined and iterated conditionals in natural language are analyzed as conditional random quantities, which can sometimes reduce to conditional events, given logical dependencies (\cite{GiSa13a,GiSa14}). For instance, the biconditional event $A||B$, which we will define by $(B | A) \wedge (A | B)$,  reduces to the conditional $(A\wedge B)|(A\vee B)$. Moreover, the notion of \emph{biconditional centering} will be given.  

The outline of the paper is as follows.
In Section~\ref{SEC:BasicNotions} we give  some preliminaries on the  notions of coherence and p-entailment for conditional random quantities, which assume values in $[0,1]$.
In Section~\ref{SEC:CENTERING}, after  recalling the notions of conjoined conditional and iterated conditional, we 
study the p-validity of centering in  the case where the basic events are replaced by conditionals.  In Section~\ref{SEC:LUBounds2PremsiseCentering} we give some results on coherence, by determining  the lower and  upper bounds for the conclusion of  two-premise centering; we  also examine  the classical case by obtaining  the  same lower and  upper bounds.
In Section~\ref{SEC:BICOND}, after recalling the classical biconditional introduction rule,  we present an analogue in terms of conditional events (\emph{biconditional AND rule});  we also obtain  one-premise and two-premise biconditional centering. In Section~\ref{SEC:LUBoundsTwoPremBiConCent} we  determine the  lower and   upper bounds for the conclusion of  two-premise biconditional centering. 
In Section~\ref{SEC:ReversedInferences} we  investigate reversed inferences (i.e., inferences from the conclusion to its premises), by determining the lower and upper bounds for the premises of  the biconditional AND rule. Section~\ref{SEC:COUNTERFACTUALS} sketches how to apply results of this paper to study selected counterfactuals, and
remark that the Import-Export Principle is not valid in our
approach which allows us to avoid Lewis' notorious triviality results.
 Section~\ref{SEC:CONCLUSION} concludes with some remarks on future work. 
Further details which expand Section~\ref{SEC:BasicNotions} are given in \ref{Appendix}.
\section{Some preliminaries}
\label{SEC:BasicNotions}
The
coherence-based approach to probability and to other uncertain measures has been adopted by many
authors (see, e.g.,
\cite{biazzo02,biazzo05,capotorti14,CaLS07,
coleti14FSS,
coletti16,
Coletti2016,coletti02,coletti15possibilistic,gilio02,gilio13ins,pfeifer13b,Wallmann2014});
we  recall below some basic aspects on the notions of coherence and of p-entailment.
In  \ref{Appendix} we will give further details on coherence of probability and prevision assessments.
\subsection{Events and constituents}
\label{sect:2.1}
In our approach  events represent uncertain facts
described by  (non ambiguous) logical propositions.   An event $A$ is
a two-valued logical entity which is either  true ($T$), or false ($F$).
The indicator of an event $A$ is a two-valued numerical quantity which is 1,
or 0, according to  whether $A$ is true, or false, respectively, and we use
 the same symbol to refer to an  event and its indicator.
We  denote by
$\Omega$ the sure event and by $\bot$ the impossible one (notice that, when necessary,
the symbol $\bot$  will denote the empty set).
Given two events $A$ and $B$,  we denote by $A\land B$ the logical intersection, or conjunction, of $A$ and $B$; moreover, we denote by  $A \vee B$ the  logical union, or disjunction, of $A$ and $B$. To simplify
notations, in many cases we  denote the conjunction of  $A$ and $B$  (and its indicator) as  $AB$;  
of course, $AB$ coincides with the product of  $A$ and $B$.
We denote by $\no{A}$ the negation of $A$. 
Of course, the truth values for conjunctions, disjunctions and
negations are obtained by applying the propositional logic. 
Given any events $A$ and $B$, we simply write $A \subseteq B$ to denote
that $A$ logically implies $B$, that is  $A\no{B}=\bot$, which
means that $A$ and $\no{B}$ cannot both be true.

Given  $n$ events
$A_1, \ldots, A_n$, as $A_i \vee \no{A}_i = \Omega \,,\;\; i = 1,
\ldots, n$,  by expanding the expression $\bigwedge_{i=1}^n(A_i \vee
\no{A}_i)$, we obtain
\[
\Omega=\bigwedge_{i=1}^n(A_i \vee
\no{A}_i)=(A_1\cdots A_n) \vee (A_1\cdots A_{n-1}\no{A}_n) \vee \cdots \vee (\no{A}_1\cdots \no{A}_n);
\]
that is the sure event $\Omega$ is represented as the disjunction of $2^n$
logical conjunctions. By discarding the conjunctions which are impossible (if any),  
 the  remaining ones are the  \emph{constituents} generated by $A_1,
\ldots, A_n$. Of course,
the constituents are pairwise logically incompatible; then, they are a partition of  $\Omega$. 
We recall that $A_1, \ldots, A_n$ are logically
independent when the number of constituents generated by them is $2^n$. Of
course, in case of some logical dependencies among $A_1, \ldots,
A_n$, the number of constituents is less than $2^n$.
For instance, given two
 events $A,B$, with $A\subseteq B$,  the constituents 
are: $AB, \no{A}B, \no{A}\no{B}$. If not stated otherwise, we assume logical independence throughout the paper. 
\subsection{Conditional events and coherent  probability assessments}
\label{sect:2.2}
Given two events $E,H$,
with $H \neq \bot$, the conditional event $E|H$
is defined as a three-valued logical entity which is true (T), or
false (F), or void (V), according to whether $EH$ is true, or $\no{E}H$
is true, or $\no{H}$ is true, respectively.
The notion of logical inclusion among events has been generalized to conditional events by Goodman and Nguyen in \cite{GoNg88} (see also \cite{gilio13ins}). Given two conditional events 
$E_1|H_1$ and $E_2|H_2$, we say that $E_1|H_1$ implies $E_2|H_2$, denoted by $E_1|H_1 \subseteq E_2|H_2$, iff $E_1H_1$ {\em true} implies $E_2H_2$ {\em true} and $\no{E}_2H_2$ {\em true} implies $\no{E}_1H_1$ {\em true}; i.e., iff $E_1H_1 \subseteq E_2H_2$ and $\no{E}_2H_2 \subseteq \no{E}_1H_1$.

 We recall that, agreeing to
the betting metaphor, if you assess $P(E|H)=p$, then, for every  real number $s$, you are willing
to pay  an amount $ps$ and  to receive $s$, or 0, or $ps$, according to
whether $EH$ is true, or $\no{E}H$ is true, or $\no{H}$ is true (bet
called off), respectively. 
Then, the random gain associated with the assessment $P(E|H)=p$ is 
$G = sH(E-p)$.

Given a real function $P : \; \mathcal{K}
\, \rightarrow \, \mathcal{R}$, where $\mathcal{K}$ is an arbitrary
family of conditional events, let us consider a subfamily
$\mathcal{F}_n = \{E_1|H_1, \ldots, E_n|H_n\}$ of 
$\mathcal{K}$, and the vector $\mathcal{P}_n =(p_1, \ldots, p_n)$,
where $p_i = P(E_i|H_i) \, ,\;\; i = 1, \ldots, n$. We denote by
$\mathcal{H}_n$ the disjunction $H_1 \vee \cdots \vee H_n$. As
$E_iH_i \vee \no{E}_iH_i \vee \no{H}_i = \Omega \,,\;\; i = 1, \ldots, n$,
by expanding the expression $\bigwedge_{i=1}^n(E_iH_i \vee \no{E}_iH_i
\vee \no{H}_i)$ we can represent $\Omega$ as the disjunction of $3^n$
logical conjunctions, some of which may be impossible. 
The
remaining ones are the constituents generated by 
 $\mathcal{F}_n$ and, of course, are a partition of $\Omega$.
We denote by $C_1, \ldots, C_m$ the constituents which logically imply 
$\mathcal{H}_n$ and (if $\mathcal{H}_n \neq \Omega$) by $C_0$ the
remaining constituent $\mathcal{\no{H}}_n = \no{H}_1 \cdots \no{H}_n$, so
that
\[
\mathcal{H}_n = C_1 \vee \cdots \vee C_m \,,\;\;\; \Omega =
\mathcal{\no{H}}_n \vee
\mathcal{H}_n = C_0 \vee C_1 \vee \cdots \vee C_m \,,\;\;\; m+1 \leq 3^n
\,.
\]
In the context of betting scheme, with the pair $(\mathcal{F}_n, \mathcal{P}_n$) we associate the random gain $G = \sum_{i=1}^n s_iH_i(E_i - p_i)$,
where $s_1, \ldots, s_n$ are $n$ arbitrary real numbers. We observe that $G$ is the difference between the amount that you receive, $\sum_{i=1}^n s_i(E_iH_i + p_i\no{H}_i)$, and the amount that you pay, $\sum_{i=1}^n s_ip_i$, and represents the net gain from engaging each transaction $H_i(E_i - p_i)$, the scaling and meaning (buy or sell) of the transaction being specified by the magnitude and the sign of $s_i$ respectively.
 Let $g_h$
be the value of $G$ when $C_h$ is true; then $G\in \mathcal{D}=\{g_0,g_1,\ldots,g_m\}$.
Of course, $g_0 = 0$. We denote by  $\mathcal{D}_{\mathcal{H}_n}$  the set of values of $G$ restricted to $\H_n$, that is $\mathcal{D}_{\mathcal{H}_n}=\{g_1, \ldots, g_m\}$.
\begin{definition}\label{COER-BET}  The function $P$ defined on $\mathcal{K}$ is said to be {\em coherent}
		if and only if, for every integer $n$, for every finite subfamily $\mathcal{F}_n$ of 
		$\mathcal{K}$ and for every  real numbers $s_1, \ldots, s_n$, one has:
		$\min  \mathcal{D}_{\mathcal{H}_n} \leq 0 \leq \max \mathcal{D}_{\mathcal{H}_n}$.
\end{definition}
Notice that the condition $\min \mathcal{D}_{\mathcal{H}_n}  \leq 0 \leq \max \mathcal{D}_{\mathcal{H}_n}$ can be written in two equivalent ways: $\min  \mathcal{D}_{\mathcal{H}_n} \leq 0$, or  $\max  \mathcal{D}_{\mathcal{H}_n} \geq 0$.  As shown by Definition \ref{COER-BET}, a probability assessment is coherent if and only if, for any finite combination of $n$ bets, it does not happen that the values $g_1, \ldots, g_m$ are all positive, or all negative ({\em no Dutch Book}).
Further technical details on coherence of probability assessments on  conditional events and  on conditional random quantities are given in \ref{Appendix}.

\subsection{Conditional random quantities and the notions of p-consistency and p-entailment}
In what follows, if not specified otherwise, we will consider conditional random quantities  which take values in a finite subset of $[0,1]$.
Based on the notions of p-consistency and p-entailment of Adams (\cite{adams75}),  which were formulated for conditional events in the setting of coherence (see, e.g., \cite{gilio10,gilio11ecsqaru,GiSa13IJAR}),  we will generalize these notions to these conditional random quantities. Let $X|H$ be a finite conditional random quantity and let  $\{x_{1}, \ldots,x_{r}\}$ denote  the set of possible values for the restriction of $X$ to $H$.
 Then,  $X|H\in[0,1]$ if and only if $x_{j}\in[0,1]$ for each $j=1,\ldots,r$; indeed  in this case  coherence requires that $\prev(X|H)\in[0,1]$ (see, e.g., \cite{GiSa14}).
 \begin{definition}\label{PC}
Let $\mathcal{F}_n = \{X_i|H_i \, , \; i=1,\ldots,n\}$ be  a family of $n$  conditional random quantities which take values in a finite subset of $[0,1]$. Then, $\mathcal{F}_n$ is  {\em p-consistent} if and only if,
the (prevision) assessment $(\mu_1,\mu_2,\ldots,\mu_n)=(1,1,\ldots,1)$ on $\mathcal{F}_n$ is coherent.
\end{definition}
\begin{definition}\label{PE}
		A p-consistent family $\mathcal{F}_{n} = \{X_{i}|H_{i} \, , \; i=1,
		\ldots ,n\}$ \emph{p-entails} a conditional random quantity $X|H$ which
		takes values in a finite subset of $[0,1]$, denoted by $\mathcal{F}
		_{n} \; \models_{p} \; X|H$, if and only if for any coherent (prevision)
		assessment $(\mu_{1},\ldots ,\mu_{n},z)$ on $\mathcal{F}_{n} \cup \{X|H
		\}$ it holds that: if $\mu_{1}=\cdots =\mu_{n}=1$, then $z=1$.

\end{definition}
Of course, when $\mathcal{F}_n$ p-entails $X|H$, there may be coherent assessments $(\mu_1,\ldots,\mu_n,z)$ with $z \neq 1$, but in such case $\mu_i \neq 1$ for at least one index $i$.
We say that the inference from $\F_n$ to $X|H$ is \emph{p-valid} if and only if  $\mathcal{F}_n \; \models_p X|H$.

Now, we  generalize  the notion of p-entailment between two finite families of conditional random quantities  $\mathcal{F}$ and $\mathcal{F}'$.
\begin{definition}\label{ENTAIL-FAM}
Given two p-consistent finite families of conditional random quantities $\mathcal{F}$ and $\mathcal{F}'$, we say that $\mathcal{F}$ p-entails $\mathcal{F}'$ if  and only if $\mathcal{F}$ p-entails $X|H$, for every $X|H \in\mathcal{F}'$.
\end{definition}
\noindent {\em Transitivity property of p-entailment:} Of course, p-entailment is transitive; that is, given three p-consistent families of conditional random quantities $\mathcal{F}, \mathcal{F}',\mathcal{F}''$, if $\mathcal{F} \; \models_p \; \mathcal{F}'$ and $\mathcal{F}' \; \models_p \; \mathcal{F}''$, then $\mathcal{F} \; \models_p \; \mathcal{F}''$.
\begin{remark}\label{ENTAIL-SUB} Notice that, from Definition \ref{PE}, we trivially have that $\mathcal{F}$ p-entails $X|H$, for every $X|H \in \mathcal{F}$; then, by Definition \ref{ENTAIL-FAM}, it immediately follows
\begin{equation}
\mathcal{F} \; \models_p \; \mathcal{F}' \;,\;\; \forall \, \mathcal{F}' \subseteq \mathcal{F} \,,\; \mathcal{F}' \neq \emptyset \,.
\end{equation}
\end{remark}
\begin{remark}
Notice that, if we consider conditional events instead of conditional random quantities, we recover the usual notions of p-consistency, p-entailment, and p-validity.
\end{remark}
\section{Centering}
\label{SEC:CENTERING}
	Given a conditional event $B|A$, if you assess $P(B|A)=x$, then for the indicator of $B|A$ we have $B|A=AB+x\widebar{A}$ (see Appendix \ref{SEC:EXTENDED}).
	Thus,  when the conditioning
	event $A$ is true then $B|A$ has the same  value as $B$ and as $AB$, while,  when the conditioning
	event $A$ is false then $B|A$ coincides with $x=P(B|A)$.
		This aspect seems related to the notion of (strong) centering
		used in Lewis' logic (\cite{lewis73}) in order to assign truth
		values to counterfactuals. In Remark~\ref{REM:BAYES} of this section we will show  that ${(B|A) \wedge A}=AB$, that is $(B|A) \wedge A$ and $AB$ are the same object; then, by the compound probability theorem, it holds that  $P[(B|A) \wedge A]=P(AB)=P(B|A)P(A)$ and then $P[(B|A) \wedge A]=P(AB)\leq P(B|A)$. This  inequality also follows by the Goodman \& Nguyen inclusion relation  $AB\subseteq B|A$ (\cite[Theorem~6]{gilio13ins}).
 We recall that the equality $P[(B|A) \wedge A]=P(B|A)P(A)$ in \cite{hajek94} is named  ``the probabilistic version of centering'' and it  has been usually looked at as a probabilistic independence  of the conditional $\textit{if } A \textit{ then } B$ from its premise $A$.  We also recall that in  \cite[footnote 5]{hajek94} Hajek and Hall observe that ``\emph{centering is a slight misnomer, since this name usually refers to a property of the nearness relation used to give the truth conditions for the conditional (each world is the nearest world to itself)''}\footnote{
 Lewis distinguished between ``centering'' and ``weak centering'' (\cite{lewis73}). We use ``centering'' simply for the equivalent of the former, ``strong'' notion   on centering. There is psychological evidence that ordinary people conform to this notion of centering for indicatives (\cite{cruz16,pfeiferTulkki17}).}. 
Interestingly, in \cite{Adams77} Adams  has shown that the Lewis theory of nearest possible worlds can be interpreted as a theory of worlds nearest in probability; in other words, according to Adams' viewpoint, the Lewis logic may be considered as  ``the logic not of truth, but of high probability''. By the previous remarks and  in agreement with \cite{hajek94} (see also \cite[p. 442]{hajek15}), we simply use the term centering also for the kind of inferences which we study in this paper.

There is \textit{one-premise centering}: inferring $\textit{if } A \textit{ then } B$ from the single premise $AB$. And \textit{two-premise centering}:
inferring $\textit{if } A \textit{ then } B$ from the two separate premises $A$ and $B$. 
Centering is valid for quite a wide range of conditionals (\cite{cruz15,cruz16,over16}).  It is clearly valid for the material conditional, since $\textit{not-}A \textit{ or } B$ must be true when  $A \textit{ and } B$ is true. It is also valid for Lewis conditional $\textit{if } A \textit{ then } B$ (\cite{lewis73}), which holds, roughly, when $B$ is true in the closest world in which $A$ is true. In \cite{lewis73} Lewis has a semantic condition of centering, which states that the actual world is the closest world to itself.
The characteristic axiom for this semantic condition is what we are also calling centering. It is probabilistically valid, \emph{p-valid}, for the conditional event, i.e. $AB$ p-entails $B|A$ and $\{A,B\}$ p-entails $B|A$.  Centering is, however, not valid for inferentialist accounts of conditionals, where an inferential relation between antecedent and consequent is presupposed (see, e.g., \cite{douven16}).

A (p-consistent) set of premises p-entails a conclusion if and only if the conclusion must have probability one when all the premises have probability one  \cite{GiSa13IJAR}. 
Clearly, one-premise centering is p-valid, indeed the p-entailment of $B|A$ from $AB$ follows by observing that $P(AB) = P(A)P(B | A)$ and so
$P(AB)\leq P(B | A)$: if $P(AB)=1$, then $P(B | A)=1$. 
 Two-premise centering is also clearly p-valid, as
it is p-valid to infer $AB$ from $A$ and $B$, and then one-premise centering can be used to infer
$B | A$:  if $P(A) = x$ and $P(B) = y$,  coherence requires that $P(AB)$ has to be in the
interval $[\max \{x + y - 1,0\}, \min\{x,y\}]$, with $P(AB)\leq P(B|A)$. Therefore,   if  $P(A)=P(B)=1$, it follows ${P(AB)=}$ $P(B|A)=1$ and then  $\{A,B\} \text{ p-entails } B|A$.

We will study the p-validity of generalized versions of one-premise and two-premise centering, where the unconditional events $A$ and $B$ are replaced by the conditional events $A|H$ and $B|K$, respectively.
These kinds of centering involve the notions of  conjunction and of iterated conditioning  for conditional events.  
Conjunction and iteration among conditionals  have been studied  from the viewpoint of random variables  by many authors (see, e.g. \cite{Jeff91,Kauf05,Kauf09,StJe94}); for an overview on conditionals, see, e.g., \cite{adams98,edgington95,Eells94}.
In our approach we exploit recent results obtained  in the setting of coherence 
for conditional random quantities  (see, e.g. \cite{GiSa13c,GiSa13a,GiSa14,GiSa17}). 
\subsection{Conjunction of two conditional events} 
We recall and discuss the notion of conjunction of two conditional events.
Note that, in numerical terms, two conditional events $A|H$ and $B|K$, with $P(A|H)=x$ and $P(B|K)=y$,  coincide with the random quantities $AH+x\no{H}$ and $BK+y\no{K}$, respectively. 
Then,  $\min \, \{A|H, B|K\}=\min \, \{AH+x\no{H}, BK+y\no{K}\}$.
\begin{definition}[Conjunction]\label{CONJUNCTION}{\rm Given any pair of conditional events $A|H$ and $B|K$, with $P(A|H) = x, P(B|K) = y$, we define their conjunction as the conditional random quantity
\[
(A|H) \wedge (B|K) = \min \, \{A|H, B|K\} \,|\, (H \vee K) =\min \, \{AH+x\no{H}, BK+y\no{K}\}\,|\, (H \vee K).
\]
}\end{definition}
Then, defining $z=\mathbb{P}[(A|H)\wedge(B|K)]$, we have
\begin{equation}\label{EQ:CONJUNCTION}
(A|H)\wedge(B|K) =\left\{\begin{array}{ll}
1, &\mbox{ if $AHBK$  is true,}\\
0, &\mbox{ if  $\no{A}H\vee \no{B}K$ is true,}\\
x, &\mbox{ if $\no{H}BK$ is true,}\\
y, &\mbox{ if $AH\no{K}$ is true,}\\
z, &\mbox{ if $\no{H}\no{K}$ is true}.
\end{array}
\right.
\end{equation}
From (\ref{EQ:CONJUNCTION}), the conjunction  $(A|H) \wedge (B|K)$ is the following  random quantity
\begin{equation}\label{EQ:REPRES}
(A|H) \wedge (B|K)=1 \cdot AHBK + x \cdot \no{H}BK + y \cdot AH\no{K} + z \cdot \no{H}\no{K}\,.
\end{equation}
 Notice that the quantity  $z=\mathbb{P}[(A|H)\wedge(B|K)]$ represents the value that you assess,  with the proviso that, for each real number $s$,  you will pay the amount $sz$ by receiving   the random quantity $s[(A|H) \wedge (B|K)]$. In particular, if $s=1$,  then you agree to pay $z$ with the proviso that you will receive: $1$, if both conditional events are true; $0$, if at least one of the conditional events is false;
$x$, if  $A|H$ is void and  $B|K$ is true;
$y$, if  $B|K$ is void and $A|H$ is true;
   $z$, if both conditional events are void.
Notice that this notion of conjunction, with positive probabilities for the conditioning events, has been already proposed in \cite{mcgee89}.
\begin{remark}\label{REM:BAYES}
We remark that in particular, given two events $A$ and $H$, with $H\neq \bot$, $P(A|H)=x$, $P(H)=y$, $\prev[(A|H)\wedge H]=z$, by (\ref{EQ:REPRES})  it holds that 
\begin{equation}
(A|H)\wedge H=(A|H)\wedge (H|\Omega)=  AHH\Omega + x \cdot \no{H}H\Omega + y \cdot AH\bot + z \cdot \no{H}\bot=
AH.
\end{equation}
Then, the conjunction $(A|H)\wedge H$ is equivalent to the unconditional  event $AH$ and $\prev[(A|H)\wedge H]=P(AH)=P(A|H)P(H)$.
\end{remark}
Notice that the notion of conjunction given in Definition~\ref{CONJUNCTION}, with \emph{positive} probabilities for the \emph{conditioning} events, has been already proposed in  the context of betting scheme in \cite{mcgee89}.
 By linearity of prevision  it holds that  
 \[
 z=P(AHBK) + x P(\no{H}BK) + y P(AH\no{K})+zP(\no{H}\no{K})\,;
 \]
in particular,  if $P(H \vee K) > 0$ we obtain the following result given  in  \cite{Kauf09,mcgee89}: 
\[
\mathbb{P}[(A|H) \wedge (B|K)] = \frac{P(AHBK) + P(A|H) P(\no{H}BK) + P(B|K) P(AH\no{K})}{P(H \vee K)} \,.
\]
We recall that a well-known notion of conjunction among conditional events, which plays an important role in nonmonotonic reasoning, is the quasi conjunction  \cite{adams75,benferhat97,GiSa13IJAR}, i.e., the following conditional event:
\[
QC(A|H,B|K)= (AH\vee \no{H})\wedge(BK\vee \no{K})| (H \vee K)\,,
\]
or in numerical terms, since $AH\vee \no{H}=AH+ \no{H}$ and $BK\vee \no{K}=BK+ \no{K}$:
\[
QC(A|H,B|K)= \min \, \{AH+\no{H}, BK+\no{K}\} \,|\, (H \vee K)\,.
\]
The event $AH\vee \no{H}$ is the material conditional associated with the conditional ``if $H$ then $A$''. Then, the quasi conjunction is defined by taking the minimum of the  material conditionals given $H\vee K$.
However,  we define the conjunction by taking the minimum of the conditional events given $H\vee K$. Our conjunction is (in general) a conditional random quantity, whereas the quasi conjunction is a conditional event. 
In some particular cases conjunction and quasi conjunction coincide; 
two cases  examined in \cite{GiSa13a} are: $(i)$ $x=y=1$; $(ii)$  $K=AH$ (or symmetrically  $H=BK$).
Moreover,  classical results concerning lower and upper bounds for the conjunction of unconditional events, which do not hold for the upper bound of the  quasi conjunction (\cite{gilio12a,gilio13ins}),
still hold for our notion of conjunction. This is shown in the next result  (\cite{GiSa14}).

\begin{theorem}\label{THM:FRECHET}{\rm
 Given any coherent assessment $(x,y)$ on $\{A|H, B|K\}$, with $A,H,B,K$ logically independent, and with $H \neq \bot, K \neq \bot$, the extension $z = \mathbb{P}[(A|H) \wedge (B|K)]$ is coherent if and only if  the Fr\'echet-Hoeffding bounds are satisfied:
\begin{equation}\label{LOW-UPPER}
max\{x+y-1,0\} = z' \; \leq \; z \; \leq \; z'' = min\{x,y\} \,.
\end{equation}
}\end{theorem}
\begin{remark}\label{REM:ANDRULE1}
We recall  that, by logical independence of $A,H,B,K$, the assessment $(x,y)$ is coherent for every  $(x,y)\in [0,1]^2$. From Theorem \ref{THM:FRECHET}, the set $\Pi$ of all coherent assessment $(x,y,z)$ on $\F=\{A|H,B|K,(A|H)\wedge(B|K)\}$ is
$\Pi=\{(x,y,z): (x,y)\in[0,1]^2, \max\{x+y-1,0\}  \leq \; z \; \leq \; \min\{x,y\}\}$. Then,  $z\in[0,1]$ and  $(A|H)\wedge (B|K)\in[0,1]$. 
Moreover,  as $(1,1,1)\in \Pi$,  the family  $\F$ (and so each subfamily of $\F$) is p-consistent.
In particular,  if $x=1,y=1$, then  $z$ must be equal to 1.
Then, by  Definition~\ref{PE}, $\{A|H,B|K\}$ p-entails $(A|H) \wedge (B|K)$, i.e., 
\begin{equation}\label{EQ:ANDRULE}
    \{A|H,B|K\} \models_p (A|H) \wedge (B|K)\, .
\end{equation}  We call this inference rule ``\emph{AND} rule for conditional events''. 
We also notice that  the assessment $(x,y,1)\in \Pi$ if and only if $x=1$ and $y=1$. Then, 
both $x=1$ and $y=1$ follow from $z=1$, i.e.  $(A|H) \wedge (B|K) \models_p \{A|H,B|K\}$, which is the converse of  (\ref{EQ:ANDRULE}).
\end{remark}
\begin{remark}\label{REM:ANDRULE}
Assuming  $HK=\bot$, it holds that the conjunction  $(A|H)\wedge (B|K)$ coincides with the product $(A|H)\cdot (B|K)$; moreover
	\[
	\prev[(A|H)\wedge (B|K)]=\prev[(A|H)\cdot (B|K)]= P(A|H)P(B|K),
	\]
which states that the random quantities  $A|H$ and $B|K$ are \emph{uncorrelated}; more details are given in   (\cite{GiSa14}).
\end{remark}
\subsection{Iterated conditioning}
We recall and discuss  the notion of iterated conditioning.
\begin{definition}[Iterated conditioning]\label{ITER-COND}{\rm
Given any pair of conditional events $A|H$ and $B|K$, the iterated conditional $(B|K)|(A|H)$ is defined as the  conditional random quantity 
$(B|K)|(A|H) = (B|K) \wedge (A|H) + \mu \no A|H$,  where $\mu=\mathbb{P}[(B|K)|(A|H)]$.
}
\end{definition}
Notice that, in the context of betting scheme, $\mu$ represents the amount you agree to pay, with the proviso that you will receive the quantity 
	\begin{equation}
		\label{EQ:ITERATED}
		(B|K)|(A|H)=\left\{
		\begin{array}{l@{\quad }l}
			1, &\mbox{ if } AHBK \mbox{ is true,}
			\\
			0, &\mbox{ if } AH\widebar{B}K \mbox{ is true,}
			\\
			y, &\mbox{ if } AH\widebar{K} \mbox{ is true,}
			\\
			\mu , &\mbox{ if } \widebar{A}H\mbox{ is true,}
			\\
			x+\mu (1-x), &\mbox{ if } \widebar{H}BK \mbox{ is true,}
			\\
			\mu (1-x), &\mbox{ if } \widebar{H}\;\widebar{B}K \mbox{ is true,}
			\\
			z+\mu (1-x), &\mbox{ if } \widebar{H}\;\widebar{K} \mbox{ is true.}
			\\
		\end{array}
		\right.
	\end{equation}
%
%\begin{equation}\label{EQ:ITERATED}
%(B|K)|(A|H)=\left\{\begin{array}{ll}
%1, &\mbox{ if }  AHBK \mbox{ true,}\\
%0, &\mbox{ if }  AH\no{B}K \mbox{ true,}\\
%y, &\mbox{ if }  AH\no{K} \mbox{ true,}\\
%\mu, &\mbox{ if }  \no{A}H\mbox{ true,}\\
%x+\mu(1-x), &\mbox{ if }  \no{H}BK \mbox{ true,}\\
%\mu(1-x), &\mbox{ if }  \no{H}\no{B}K \mbox{ true,}\\
%z+\mu(1-x), &\mbox{ if }  \no{H}\no{K} \mbox{ true.}\\
%\end{array}
%\right.
%\end{equation}
We recall the following product formula (\cite{GiSa13a}).
\begin{theorem}[Product formula]\label{THM:PRODUCT}{\rm
Given any  assessment $x=P(A|H), \mu=\mathbb{P}[(B|K) | (A|H)]$, $z=\mathbb{P}[(B|K) \wedge (A|H)]$, if $(x,\mu,z)$ is coherent, then
 $z=\mu \cdot x$, i.e.,
 \begin{equation}\label{EQ:PRODUCTFORMULA}
 \mathbb{P}[(B|K) \wedge (A|H)] = \mathbb{P}[(B|K) | (A|H)]P(A|H) \,.
 \end{equation}
}
\end{theorem}
As  $z=\mu x$, it follows that $z+\mu(1-x)=\mu$. Then, from (\ref{EQ:ITERATED}), $(B|K)|(A|H)$ coincides with 
\[
AHBK+yAH\no{K}+(x+\mu(1-x))\,\no{H}BK+\mu(1-x)\,\no{H}\no{B}K+\mu\,(\no{A}H \vee 
\no{H}\no{K}).
\]
\begin{remark}
As $x\geq z$, for $x>0$  one has $\mu=\frac{z}{x}\in[0,1]$; moreover  $x+\mu(1-x)$ is a linear convex combination of the values $\mu$ and $1$, then  $x+\mu(1-x)\in[\mu,1]$. Therefore, for $x>0$, $(B|K)|(A|H)\in [0,1]$. As shown in Theorem~\ref{THM:CENTERINGLU}, $\mu\in[0,1]$ also for $x=0$. Thus,  $(B|K)|(A|H)\in[0,1]$ in all cases.
\end{remark}
\subsection{One-premise and two-premise centering: p-validity}
The \emph{one-premise centering} involving conditional events is represented by the following inference rule: \emph{from  $(A|H)\wedge (B|K)$ infer $(B|K)|(A|H)$}.
Likewise, \emph{two-premise centering} involving conditional events is represented by: \emph{from  $\{A|H, B|K\}$ infer $(B|K)|(A|H)$}.
Are these inference rules p-valid?

One-premise centering is p-valid;
indeed, from (\ref{EQ:PRODUCTFORMULA}) it holds that
\begin{equation}\label{EQ:INEQ}
\mathbb{P}[(B|K) \wedge (A|H)]\leq \mathbb{P}[(B|K) | (A|H)],
\end{equation}
then  $\mathbb{P}[(B|K) \wedge (A|H)]=1$ implies  $\mathbb{P}[(B|K) | (A|H)]=1$, i.e.,  
\begin{equation}\label{EQ:1CENTERING}
(B|K) \wedge (A|H)\,\, \models_p  \,\, (B|K) | (A|H).
\end{equation}

Two-premise centering is also  p-valid; 
indeed,  from (\ref{EQ:ANDRULE})  and 
(\ref{EQ:1CENTERING}), by transitivity, 
\begin{equation}\label{EQ:2CENTERING}
\{(A|H), (B|K)\}\,\, \models_p  \,\, (B|K) | (A|H),  
\end{equation}
that is, if $P(A|H)=1$ and $P(B|K)=1$, then  $\mathbb{P}[(B|K) | (A|H)]=1$. 
\section{Lower and upper bounds for two-premise centering}
\label{SEC:LUBounds2PremsiseCentering}
In this section we  give a probabilistic analysis of two-premise centering by determining the coherent lower and upper bounds for the conclusion.
We first consider  the general case:  \emph{from  $\{A|H, B|K\}$ infer $(B|K)|(A|H)$}. Then, we consider two-premise centering with unconditional events in the premise set: \emph{from  $\{A, B\}$ infer $B|A$}, which is a particular case where $H=K=\Omega$.
\subsection{The general case: ``from  $\{A|H, B|K\}$ infer $(B|K)|(A|H)$''}
We start by computing the set of all coherent assessments on the elements of centering.
\begin{theorem}\label{THM:PIONITERATED}
	Let $A,B,H,K$ be any logically independent events. 
	The set  $\Pi$ of  all coherent assessments $(x,y,z,\mu)$ on the family $\F=\{A|H,B|K$, $(A|H)\wedge (B|K), (B|K)|(A|H)\}$ is $\Pi=\Pi'\cup \Pi''$, where 
\begin{equation}\label{EQ:PI}
\begin{array}{ll}
%\Pi'=\{(x,y,z,\mu): x\in ]0,1], y\in[0,1],  z\in [\max\{x+y-1,0\}, \min\{x,y\}], \mu=\frac{z}{x}\},\\
\Pi'=\{(x,y,z,\mu): x\in (0,1], y\in[0,1], 
z\in [z', z''], \mu=\frac{z}{x}\},  \\
\text{with }  z'=\max\{x+y-1,0\}, z''= \min\{x,y\}, \text{ and}\\
\Pi''=\{(0,y,0,\mu): (y,\mu)\in[0,1]^2\}. 
\end{array}
\end{equation}
\end{theorem}
\begin{proof}
We recall that the assessment $(x,y)$ on $\{A|H,B|K\}$ is coherent for every $(x,y)\in[0,1]^2$. By Theorem \ref{THM:FRECHET}, the assessment  $z=\mathbb{P}[(A|H)\wedge(B|K)]$ is a coherent extension of $(x,y)$ if and only if  $z\in[z',z'']$, where 
$z'=\max \{x+y-1,0\}$ and $z''=\min  \{x,y\}$. 
Moreover, assuming $x>0$, by Theorem \ref{THM:PRODUCT} it holds that  $\mu=\frac{z}{x}$. Then, 
every $(x,y,z,\mu)\in \Pi'$ is coherent, that is $\Pi'\subseteq \Pi$. Of course, if $x>0$ and $(x,y,z,\mu)\notin \Pi'$, then $(x,y,z,\mu)$ is not coherent. 
Now, we assume  $x=0$, so that  $z'=z''=0$. Then,  we show that the assessment $(0,y,0,\mu)$ is coherent if and only if $(y,\mu)\in[0,1]^2$, that is $(0,y,0,\mu)\in \Pi''$. 
As $x=0$, it holds that  $A|H=AH+x\no{H}=AH$. Then, $(B|K)|(A|H)=(B|K)|AH=(BK+y\no{K})|AH$ and $\F=\{A|H,B|K,(A|H)\wedge (B|K), (BK+y\no{K})|AH\}$.
 The constituents $C_h$'s and  the  points $Q_h$'s associated with $(\F,\M)$, where  $\M=(0,y,0,\mu)$, are given in Table \ref{TAB:TABLE}.
\begin{table}[h]
	\centering
\begin{tabular}{|L|L|L|L|L}
	\hline
	    & C_h            & Q_h                          &  \\
	\hline
	C_1 & AHBK           & (1,1,1,1)                    & Q_1   \\
	C_2 & AH\no{B}K      & (1,0,0,0)                    & Q_2   \\
	C_3 & AH\no{K}       & (1,y,y,y)                    & Q_3   \\
	C_4 & \no{A}HBK      & (0,1,0,\mu)                    & Q_4   \\
	C_5 & \no{A}H\no{B}K & (0,0,0,\mu)                    & Q_5   \\
	C_6 & \no{A}H\no{K}  & (0,y,0,\mu)                  & Q_6   \\
	C_7 & \no{H}BK       & (0,1,0,\mu)                    & Q_7   \\
	C_8 & \no{H}\no{B}K  & (0,0,0,\mu)                    & Q_8   \\
	C_0 & \no{H}\no{K}   &(0,y,0,\mu)                       & Q_0=\mathcal{M}   \\
	\hline
\end{tabular}
	\caption{Constituents $C_h$'s and  points $Q_h$'s associated with  the prevision  assessment    $\mathcal{M}=(0,y,0,\mu)$  on 
		$\F=\{A|H,B|K,(A|H)\wedge (B|K), (BK+y\no{K})|AH\}$.
	}
	\label{TAB:TABLE}
\end{table}
Denoting by  $\mathcal{I}$ be the convex hull generated by  $Q_1,Q_2, \ldots,Q_{8}$, the coherence of the prevision assessment $\mathcal{M}$ on $\F$ requires that the condition $\P\in \mathcal{I}$ be satisfied; this amounts to the solvability of the following system 
\begin{equation}\label{eqn:sigma}
\begin{array}{l}
 \hspace{1 cm}
\M=\sum_{h=1}^{8} \lambda_hQ_h,\;\;\;
\sum_{h=1}^{8} \lambda_h=1,\;\;\; \lambda_h\geq 0,\,  \; h=1,\ldots,8 \,.
\end{array}
\end{equation}
		As $\mathcal{M}=yQ_{4}+(1-y)Q_{5}$, the vector $(\lambda_{1},\ldots ,
		\lambda_{8})=(0,0,0,y,1-y,0,0,0)$ is a solution of
		system~(\ref{eqn:sigma}) such that $\sum_{h:C_{h}\subseteq HK}\lambda
		_{h}=\lambda_1+\lambda_2+\lambda_4+\lambda_5=\lambda_4+\lambda_5=1>0$, so that $\sum_{h:C_{h}\subseteq H}\lambda_{h}=\lambda_1+\cdots+\lambda_6=1>0$ and
		$\sum_{h:C_{h}\subseteq K}\lambda_{h}=\lambda_1+\lambda_2+\lambda_4+\lambda_5+\lambda_7+\lambda_8=1>0$; while, $
		\sum_{h:C_{h}\subseteq AH}\lambda_{h}=\lambda_1+\lambda_2+\lambda_3=0$. Then, by (\ref{EQ:I0}),
		$\mathcal{I}_{0}\subseteq \{4\}$ and $\mathcal{F}_{0}\subseteq \{(BK+y
		\widebar{K})|AH\}$. Thus, from Theorem~\ref{CNES-PREV-I_0-INT}, for
		checking coherence of $\mathcal{M}$ on $\mathcal{F}$ it is sufficient
		to study the coherence of $\mu =\mathbb{P}(\{(BK+y\widebar{K})|AH\})$.
		The random gain for the assessment $\mu $ is
\[
G=sAH(BK+y\no{K}-\mu),\;\; s\in \mathbb{R}\,.
\]
Without loss of generality, we can assume $s=1$.
The constituents contained in $AH$ are: 
$C_1=AHBK, C_2=AH\no{B}K, C_3=AH\no{K}$. The corresponding values for the random gain $G$ are: 
$g_1=(1-\mu),\; g_2=-\mu,\; g_3=y-\mu$. Then, the set of  values of $G$ restricted to $AH$ is $\mathcal{D}_{AH}=\{g_1,g_2,g_3\}$.
As it can be verified, 
\[
\begin{array}{ll}
\min\mathcal{D}_{AH}>0 \;\Longleftrightarrow \; \mu <0,\;\;
\max\mathcal{D}_{AH}<0 \;\Longleftrightarrow \; \mu >1\,.
\end{array}
\]
Therefore, the  condition of coherence on $\mu$, that is $\min \mathcal{D}_{AH}\cdot \max \mathcal{D}_{AH}\leq 0$, is satisfied if and only if $\mu\in[0,1]$.
Thus, every assessment $(0,y,0,\mu)$ is coherent if and only if $(0,y,0,\mu)\in \Pi''=\{(0,y,0,\mu): (y,\mu)\in[0,1]^2\}$. Therefore $\Pi=\Pi'\cup \Pi''$.
\end{proof}
Based on Theorem~\ref{THM:PIONITERATED}, we obtain the following prevision propagation rule for two-premise centering: 
\begin{theorem}\label{THM:CENTERINGLU}
	Let $A,B,H,K$ be any logically independent events. Given a coherent assessment $(x,y)$ on $\{A|H,B|K\}$, for the iterated conditional $(B|K)|(A|H)$ the  extension $\mu=\prev((B|K)|(A|H))$ is coherent if and only if
	$\mu \in [\mu', \mu'']$, where 
	\begin{equation}\label{EQ:CENTERINGLU}
	\begin{array}{ll}
	\mu'=\left\{
	\begin{array}{ll}
	\max\left\{0,\tfrac{x+y-1}{x}\right\},& \text{ if } x>0;\\
	0,& \text{ if } x=0;\\
	\end{array}
	\right.\;\;
	\mu''=\left\{
	\begin{array}{ll}
	\min\left\{1,\tfrac yx\right\},& \text{ if } x>0;\\
	1,& \text{ if } x=0.\\
	\end{array}
	\right.
	\end{array}
	\end{equation} 
\end{theorem}
\begin{proof}
Assume that  $x=0$. From Theorem \ref{THM:PIONITERATED} it follows that the set of all coherent assessments $(x,y,z,\mu)$ on $\F=\{A|H,B|K$, $(A|H)\wedge (B|K), (B|K)|(A|H)\}$ is  $\Pi''=\{(0,y,0,\mu): (y,\mu)\in[0,1]^2\}$. Then, 
$\mu$ is a coherent extension of $(x,y)$ if and only if $\mu \in[\mu',\mu'']$, where $\mu'=0$ and $\mu''=1$.

Assume that $x>0$. From Theorem \ref{THM:PIONITERATED} 
it follows that
		the set of all coherent assessments $(x,y,z,\mu )$ on $\mathcal{F}$ is
		$\Pi '=\{(x,y,z,\mu ): 0<x\leq 1, 0\leq y\leq 1, z' \leq z\leq z'',
		\mu =\frac{z}{x}\} $, where $z'=\max \{x+y-1,0\}$ and $z''=\min \{x,y
		\}$. Then, $\mu $ is a coherent extension of $(x,y)$ if and only if
		$\mu \in [\mu ', \mu '']$, where $\mu '=\frac{z'}{x}=\max \left\{ 
		\tfrac{x+y-1}{x},0\right\} $ and $\mu ''=\frac{z''}{x}=\min \left\{
		\tfrac{y}{x},1\right\} $.
\end{proof}
\begin{remark}
The p-validity of  two-premise centering given in (\ref{EQ:2CENTERING}) directly follows as an instantiation of 
 Theorem~\ref{THM:CENTERINGLU} with  $x=1$ and $y=1$. 
\end{remark}
\subsection{The case $H=K=\Omega$}
In case of logical dependencies among events, as we know,  the set of all coherent assessments may be smaller than the set given  in Theorem~\ref{THM:PIONITERATED}.
We examine the case $H=K=\Omega$, by showing that the set $\Pi$  of all coherent assessments  on $\F=\{A,B, AB, B|A\}$ is still the same as in  Theorem~\ref{THM:PIONITERATED}.
\begin{theorem}\label{THM:PI}
	Let $A,B$ be any logically independent events. 
	The set  $\Pi$ of  all coherent assessments $(x,y,z,\mu)$ on the family $\F=\{A,B, AB, B|A\}$ is $\Pi=\Pi'\cup \Pi''$, with $\Pi'$ and $\Pi''$ as defined in formula (\ref{EQ:PI}).
\end{theorem}
\begin{proof}
We recall that the assessment $(x,y)$ on $\{A,B\}$ is coherent for every $(x,y)\in[0,1]^2$.  The assessment  $z=P(AB)$ is a coherent extension of $(x,y)$ if and only if  $z\in[z',z'']$, where 
$z'=\max \{x+y-1,0\}$ and $z''=\min  \{x,y\}$. 
Moreover, assuming $x>0$, by compound probability theorem  it holds that  $\mu=\frac{z}{x}$. Then, 
every $(x,y,z,\mu)\in \Pi'$ is coherent, that is $\Pi'\subseteq \Pi$. Of course, if $x>0$ and $(x,y,z,\mu)\notin \Pi'$, then $(x,y,z,\mu)$ is not coherent.  Now, we assume  $x=0$, so that  $z'=z''=0$. Then,  we show that the assessment $(0,y,0,\mu)$ is coherent if and only if $(y,\mu)\in[0,1]^2$, that is $(0,y,0,\mu)\in \Pi''$. 
 The constituents $C_h$'s and  the  points $Q_h$'s associated with $(\F,\mathcal{P})$, where  $\mathcal{P}=(0,y,0,\mu)$, are given in Table \ref{TAB:TABLENOIT}.
\begin{table}[h]
	\centering
\begin{tabular}{|L|L|L|L|L}
	\hline
	    & C_h            & Q_h                          &  \\
	\hline
	C_1 & AB           & (1,1,1,1)                    & Q_1   \\
	C_2 & A\no{B}      & (1,0,0,0)                    & Q_2   \\
	C_3 & \no{A}B      & (0,1,0,\mu)                    & Q_3   \\
	C_4 & \no{A}\no{B} & (0,0,0,\mu)                        & Q_4   \\
	\hline
\end{tabular}
	\caption{Constituents $C_h$'s and  points $Q_h$'s associated with  the probability  assessment    $\mathcal{P}=(0,y,0,\mu)$  on 
		$\F=\{A,B,AB, B|A\}$.
	}
	\label{TAB:TABLENOIT}
\end{table}
Denoting by  $\mathcal{I}$ the convex hull generated by  $Q_1,Q_2,Q_3,Q_4$, the coherence of the prevision assessment $\mathcal{P}$ on $\F$ requires that the condition $\P\in \mathcal{I}$ be satisfied; this amounts to the solvability of the following system 
\begin{equation}\label{eqn:sigma2}
\begin{array}{l}
\mathcal{P}=\sum_{h=1}^{4} \lambda_hQ_h,\;\;\;
\sum_{h=1}^{4} \lambda_h=1,\;\;\; \lambda_h\geq 0,\,  \; h=1,\ldots,4 \,.
\end{array}
\end{equation}
As $\mathcal{P}=yQ_3+(1-y)Q_4$, the vector $(\lambda_1,\ldots,\lambda_4)=(0,0, y,1-y)$ is a solution of system (\ref{eqn:sigma2}), with
 $\sum_{h:C_h\subseteq A}\lambda_h=0$. Then, by (\ref{EQ:I0}), $\I_0\subseteq\{4\}$ and $\F_0\subseteq \{B|A\}$. Thus, from  Theorem~\ref{CNES-PREV-I_0-INT},
for checking coherence of $\mathcal{P}$ on $\F$ it is sufficient to study the coherence of $\mu=P(B|A)$.  Of course, $\mu=P(B|A)$ is coherent if and only if $\mu\in[0,1]$.
Thus, every assessment $(0,y,0,\mu)$ is coherent if and only if $(0,y,0,\mu)\in \Pi''=\{(0,y,0,\mu): (y,\mu)\in[0,1]^2\}$. Therefore $\Pi=\Pi'\cup \Pi''$.
\end{proof}

Based on Theorem~\ref{THM:PI}, we obtain the following prevision propagation rule for two-premise centering with unconditional events in the premise set: 
\begin{theorem}\label{THM:BASICCENTERINGLU}
	Let $A,B$ be any logically independent events. Given a coherent assessment $(x,y)$ on $\{A,B\}$, for the conditional event $B|A$ the  extension $\mu=P(B|A)$ is coherent if and only if
	$\mu \in [\mu', \mu'']$, where \[
	\begin{array}{ll}
	\mu'=\left\{
	\begin{array}{ll}
	\max\left\{\tfrac{x+y-1}{x},0\right\},& \text{ if } x>0;\\
	0,& \text{ if } x=0;\\
	\end{array}
	\right.\;\;
	\mu''=\left\{
	\begin{array}{ll}
	\min\left\{\tfrac yx,1\right\},& \text{ if } x>0;\\
	1,& \text{ if } x=0.\\
	\end{array}
	\right.
	\end{array}
	\]
\end{theorem}
\begin{proof}
The proof is the same of Theorem \ref{THM:CENTERINGLU}, 
with $\F=\{A,B,AB,B|A\}$ and 
with Theorem~\ref{THM:PIONITERATED}  replaced by Theorem~\ref{THM:PI}.
\end{proof}
\begin{remark}
As shown by theorems  \ref{THM:CENTERINGLU} and \ref{THM:BASICCENTERINGLU}, 
 the  lower and upper bounds on the conclusion of   two-premise centering involving iterated conditionals coincide with the respective bounds on the conclusion of  the (non-iterated) two-premise centering. 
\end{remark}
\section{Biconditional centering}
\label{SEC:BICOND}
In classical logic the biconditional $A\leftrightarrow B$ (defined by $\no{(A \vee B)} \vee (AB)$)
can be represented by the conjunction of the two material conditionals $\no A\vee B$ and $\no B\vee A$.
Therefore, $\{\no A\vee B,\no B\vee A\}\models A\leftrightarrow B$\footnote{
We recall that the symbol $\models$ denotes the relation of logical entailment. In our case, if both events  $\no A\vee B$ and  $\no B\vee A$ are true, then the biconditional $A\leftrightarrow B$ is true.
}, which is called \emph{biconditional introduction rule}. 
With the material conditional interpretation of a conditional, the biconditional  $A \leftrightarrow B$ represents the conjunction of the two conditionals $\textit{if } A \textit{ then }
B$ and $\textit{if } B \textit{ then }
A$. 
In this section, we present an analogue in terms of conditional events, by also giving a meaning 
 to the conjunction of two conditional events $A|B$ and $B|A$. 

From centering it follows that $\{A,B\} \models_p B|A$ and $\{A,B\} \models_p A|B$. 
Then, from $P(A)=P(B)=1$ it follows that $P(B|A)=P(A|B)=1$, which we denote by: $\{A,B\} \models_p \{A|B, B|A\}$.
Thus, by applying~(\ref{EQ:ANDRULE}) with $H=B$ and $K=A$, we obtain $\{A|B,B|A\} \models_p (A|B) \wedge (B|A)$ (which we call \emph{biconditional introduction} rule, or \emph{biconditional AND} rule). Then,  by transitivity
\begin{equation}\label{EQ:ANDRULEBICONDITIONAL}
\{A,B\}\; \models_p\; (A|B) \wedge (B|A)\,.
\end{equation}
In a similar way, we can prove that
\begin{equation}\label{EQ:ANDRULEBICONDITIONAL2}
AB\; \models_p\; (A|B) \wedge (B|A)\,.
\end{equation}
We recall that the conditional event  $(AB) \,|\, (A \vee B)$, denoted by $A||B$, captures the notion of the \emph{biconditional event}, which has been seen as the conjunction of two conditionals with the same truth table as the ``defective'' biconditional discussed in \cite{gauffroy09}; see also \cite{fugard11a}.
 We have
\begin{theorem}\label{THM:BIC}
Given two events $A$ and $B$ it holds that: $(A|B)\wedge (B|A)=(AB)|(A\vee B)=A||B$.
\end{theorem}
\begin{proof}
We note that $
(A|B)\wedge (B|A)=\min(A|B,B|A)|(A\vee B)=AB+\mu\cdot \no{A}\no{B}$,
 where $\mu=\mathbb{P}[(A|B)\wedge (B|A)]$; we also observe  that $(AB)|(A\vee B)=AB+p\cdot \no{A} \,\no{B}$, where $p=P[(AB)|(A\vee B)]$. Then,  under the assumption that ``$(A\vee B)$ is true'', the two random quantities $(A|B)\wedge (B|A)$ and $(AB)|(A\vee B)$ coincide. 
By coherence (see \cite[Theorem~4]{GiSa14}) it follows that these two random quantities  coincide also under the assumption  that ``$(A\vee B)$  is false'', that is
$\mu$ and $p$  coincide.  Therefore,  $(A|B)\wedge (B|A)=(AB)|(A\vee B)$. 
\end{proof}
Based on Theorem \ref{THM:BIC}, we can now really interpret the biconditional event $A||B$ as the conjunction of the two conditionals $(B | A)$ and   $(A | B)$.  Moreover, equations (\ref{EQ:ANDRULEBICONDITIONAL}) and (\ref{EQ:ANDRULEBICONDITIONAL2}) represent what we call \emph{two-premise biconditional centering} and \emph{one-premise biconditional centering} respectively, that is 
$\{A,B\} \models_p A||B$ and $AB \models_p A||B$. 

Though in classical logic $\{\no A, \no B\} \models (A\leftrightarrow B)$, the  analogue  does not hold in our approach,
since we do not have p-entailment of $A||B$ from $\no A, \no B$, indeed  if $P(\no A)=P(\no B)=1$, then $P(A\vee B)=0$ and therefore $P(A||B)=P((AB)|(A\vee B))\in[0,1]$ (see Theorem \ref{THM:BICONDITIONALLOWERUPPER} below).
The biconditional event $A||B$ is of interest to psychologists
because there is evidence that children go through a developmental stage in which they judge that $P(\textit{if } A \textit{ then }
B) = P[(AB )|(A \vee B)]$, with this judgment being replaced by \textit{P(if A then B)} = $P(B
| A)$ as they grow older (\cite{gauffroy09}).
We recall that, given two conditional events $A|H$ and $B|K$, their quasi conjunction is defined as the conditional event $Q(A|H,B|K)=[(AH\vee\no{H})\wedge (BK\vee \no{K})]|(H\vee K)$.
Quasi conjunction is a  basic notion in the work of Adams (\cite{adams75}) 
and plays a role in characterizing entailment from a conditional knowledge base (see also \cite{benferhat97}).
We recall that in \cite{gilio13ins}  $A || B$ was interpreted by the quasi conjunction  of $A|B$ and $B|A$, by obtaining  $A||B=Q(A|B,B|A)=(AB)|(A\vee B)$. 
\section{Lower and upper bounds for  two-premise biconditional centering}
\label{SEC:LUBoundsTwoPremBiConCent}
In this section we determine the lower and upper bounds 
 for the conclusion of    two premise biconditional centering.\footnote{Coherence of probability assessments on conditional events can be  checked, for example, by the CkC-package \cite{capotorti07}.}
\begin{theorem}\label{THM:BICONDITIONALLOWERUPPER}
Let $A,B$ be any logically independent events. Given any (coherent) assessment $(x,y)\in[0,1]^2$ on $\{A,B\}$, for the biconditional event  $A||B$ the  extension $z =P(A||B)$ is coherent if and only if
$z \in [z', z'']$, 
where 
\begin{equation}
\begin{array}{ll}
z'=\max\left\{x+y-1,0\right\},\;\;
z''=\left\{
\begin{array}{ll}
\frac{\min\{x,y\}}{\max\{x,y\}},& \text{ if } x>0 \text{ or } y>0 ,\\
1,& \text{ if } x=0 \text{ and } y=0 .\\
\end{array}
\right.
\end{array}
\end{equation} 
\end{theorem}
\begin{proof}
We consider two cases: $(i)$ $x>0$ or $y>0$; $(ii)$ $x=0$ and $y=0$.

Case $(i)$. As $P(A\vee B)\geq \max\{x,y\}$, it follows that $P(A\vee B)>0$. Then, defining $\nu=P(AB)$, one has $P(A||B)=\frac{P(AB)}{P(A\vee B)}=\frac{\nu}{x+y-\nu}$. 
 We recall that $\nu$ is  a coherent extension $(x,y)$ if and only if $\nu\in[\nu',\nu'']$, where $\nu'=\max\{x+y-1,0\}$ and $\nu''=\min\{x,y\}$. By observing that $f(\nu)=\frac{\nu}{x+y-\nu}$ is an increasing function of $\nu$, it follows that the assessment $z=P(A||B)$ 
 is a coherent extension of $(x,y)$ if and only  if 
 $z\in[z',z'']$, where $z'=\frac{\nu'}{x+y-\nu'}=\frac{\max\{x+y-1,0\}}{x+y-\max\{x+y-1,0\}}=\frac{\max\{x+y-1,0\}}{\min\{x+y,1\}}=\max\{x+y-1,0\}$, and 
$z''=\frac{\nu''}{x+y-\nu''}=\frac{\min\{x,y\}}{\max\{x,y\}}$. We observe that if $x=0$ or $y=0$, then $z''=z'=0$; if $x=y=1$, then $z'=z''=1$. Moreover, if $x>0$ and $y>0$, then $z''=\min\{\frac{x}{y},\frac{y}{x}\}$.

		Case $(ii)$ ($x=y=0$). The constituents $C_{h}$'s, $h=1,2,3,4$, associated
		with the assessment $(0,0,z)$ on $\{A,B,{AB|(A\vee B)}\}$, and the
		corresponding points $Q_{h}$'s are $C_{1}=AB, C_{2}=A\widebar{B}, C_{3}=
		\widebar{A}B, C_{4}=\widebar{A}\;\widebar{B}$, and $Q_{1}=(1,1,1), Q
		_{2}=(1,0,0), Q_{3}=(0,1,0), Q_{4}=(0,0,z)$, respectively. As the
		prevision point $(0,0,z)$ coincides with $Q_{4}$, then it belongs to the
		convex hull of points $Q_{1},\ldots ,Q_{4}$, that is the associated
		system $(\Sigma )$, as defined in Section~\ref{sect:2.2}, is solvable.
		As $P(A\vee B)\leq \min \{x+y,1\}=0$, each solution $(\lambda_{1},
		\ldots ,\lambda_{4})$ of $(\Sigma )$ is such that $
		\sum_{h: C_{h}\subseteq A \vee B}\lambda_{h}=\lambda_1+\lambda_2+\lambda_3=0$, so that $\mathcal{I}
		_{0}=\{3\}$. By Theorem~\ref{CNES-PREV-I_0-INT}, $(0,0,z)$ is coherent
		if and only if $z$ is coherent, which amounts to $z\in [0,1]$.

\end{proof}
\section{Reversed inferences and  bounds on biconditional AND rule}
\label{SEC:ReversedInferences}
In this section we first recall the lower and upper bounds on the conclusion of the biconditional AND rule. Then,   we study the reverse inferences from the prevision assessment on the conclusion $A||B=(A|B)\wedge (B|A)$ to the premises  $\{A|B, B|A\}$. That is, starting with a given assessment  $z \in [0,1]$ on $A||B$,  we determine the set $D_z$ of all coherent extensions $(x,y)$, where $x=P(A|B)$ and $y=P(B|A)$.
We recall the following probabilistic propagation rule (\cite{gilio13ins}).  Let  $(x,y)$ be  any coherent assessment on $\{A|B,B|A\}$; then, the probability assessment $z = P(A||B)$ is
a coherent extension of $(x,y)$ if and only if
\begin{equation}\label{EQ:PROPRULEBICONDTIONAL}
z=T^{H}_{0}(x,y)=\left\{\begin{array}{ll}
0, & \text{ if } x=0 \text{ or } y=0 \,,\\
\frac{xy}{x+y-xy}=\tfrac{1}{\frac{1-x}{x}+\frac{1-y}{y}+1}, &  \text{ if } 0<x\leq 1 \text{ and }  0<y\leq 1\,,
\end{array}\right.
\end{equation}
where  $T_0^H(x,y)$ is  the Hamacher t-norm, with parameter $\lambda=0$.
We obtain
\begin{theorem}\label{THM:Dz}
Let $A,B$ be any logically independent events. Given any assessment $z\in[0,1]$ on $A||B$, the  extension $(x,y)$ on $\{A|B,B|A\}$ is coherent if and only if $(x,y)\in D_z$, where 
\[
D_z=
\left\{
\begin{array}{ll}
 \{(x,y)\in[0,1]^2: x=0 \text{ or } y=0\}, & \text{ if } z=0,\\
 \{(x,y)\in[0,1]^2:z\leq x\leq 1, y=\frac{xz}{x-z+xz}\}, & \text{ if } 0<z\leq 1\,.
\end{array}
\right.
\]
\end{theorem}
\begin{proof}
From (\ref{EQ:PROPRULEBICONDTIONAL}) the set $\Pi$ of all coherent assessments $(x,y,z)$
 on $\{A|B,B|A, A||B\}$ is 
 \begin{equation}\label{EQ:PI2}
 \Pi=\{(x,y,z): (x,y)\in[0,1]^2, z=T_0^H(x,y)\}.
 \end{equation}
Assume that  $z=0$. We notice that  the assessment $(x,y,0)\in \Pi$ if and only if 
 $x=0$ or $y=0$, with $(x,y)\in[0,1]^2$. Then, $D_0=
\{ (x,y)\in[0,1]^2: x=0 \text{ or } y=0\}$.
Assume that $0<z\leq 1$.
By Goodman and Nguyen   inclusion relation among conditional events, as
 $ AB|(A \vee B) \subseteq A|B$ and $ AB|(A \vee B)
\subseteq B|A$,   coherence requires that (see, e.g., \cite[Theorem~6]{gilio13ins})  $x \geq z > 0$ and $y \geq z > 0$; thus
$xy > 0$, $x+y-xy> 0$, and $x-z+xz>0$. Then, from (\ref{EQ:PROPRULEBICONDTIONAL}) and  (\ref{EQ:PI2}) it holds that $z=\frac{xy}{x+y-xy}$, so that  $y=\frac{xz}{x-z+xz}$.
Therefore, $D_z=\{(x,y):z\leq x\leq 1, y=\frac{xz}{x-z+xz}\}$.
\end{proof}
\begin{remark}
Based on Theorem \ref{THM:Dz},
 the set $\Pi$ of all coherent assessments $(x,y,z)$
on $\{A|B,B|A, A||B\}$ can also be written as
\[
\Pi=\{(x,y,z): z\in[0,1], (x,y)\in D_z\}.
\]
Moreover, by symmetry, we observe that, if $z>0$, the set $D_z$ in Theorem \ref{THM:Dz} can also be written as  $D_z=\{(x,y)\in[0,1]^2:z\leq y\leq 1, x=\frac{yz}{y-z+yz}\}$.
\end{remark}
\section{Two-premise centering with logical relations and counterfactuals}
\label{SEC:COUNTERFACTUALS}
In this section we consider an instance of two premise-centering, with a logical dependency, that can be used to study some counterfactuals.
Specifically,  we consider the inference: $\{B|\Omega,C|A\}$ p-entails $(C|A)|(B|\Omega)$, with  $AB=\bot$.
As $B|\Omega=B$, this inference can be simply written as 
\begin{equation}
\{B,C|A\}\models_p (C|A)|B, \;\; \text{ with }  AB=\bot.
\end{equation}
We first show that, assuming $P(B)>0$,  the prevision of the conclusion $(C|A)|B$  coincides just with   $P(C|A)$, i.e.,  
$\prev[(C|A)|B]=P(C|A)$.
By (\ref{EQ:REPRES}) the conjunction of $B$ and $C|A$ reduces to   
the random quantity
$(C|A)| B =yB$, where   $y=P(C|A)$. 
Then, by linearity of the  prevision, $\mathbb{P}[(C|A)\wedge B]=P(C|A)P(B)$.
Moreover, by (\ref{EQ:PRODUCTFORMULA}), it holds that
$\mathbb{P}[(C|A)\wedge B]=\mathbb{P}[(C|A)|B]P(B)$ and then, by assuming $P(B)>0$,  we obtain
\begin{equation}\label{EQ:INCOMP}
\mathbb{P}[(C|A)|B]=\frac{\mathbb{P}[(C|A)\wedge B]}{P(B)}=\frac{P(C|A)P(B)}{P(B)}=P(C|A)\,.
\end{equation}
Now we show that  (\ref{EQ:INCOMP}) holds in general, even if $P(B)=0$, by also showing that 
the iterated conditional $(C|A)|B$ is constant and coincides with $P(C|A)$, when $AB=\bot$.
%(see \cite[Example 1]{GiSa13c})
As $(C|A)\wedge B=yB$,   by Definition \ref{ITER-COND}, $(C|A)|B=yB+\mu\no{B}$.
Moreover, 
as $B\subseteq \no{A}$, {\em conditionally on $B$ being true}, it holds that: $C|A=AC+y\no{A}=y$; 
   that is, when $B$ is true,  $C|A$ is constant and equal to $y$. Then, by coherence, 
$\mu = \pr[(C|A)|B]=\pr[(AC+y\no{A})|B] =\pr(y|B)= y$ (see  \cite[Remark 1]{GiSa13c}). Therefore, when $AB=\bot$ it holds that
$(C|A)|B=yB+y\no{B}=y$, i.e., the iterated conditional $(C|A)|B$ is constant and equal to  $P(C|A)$. Then,
trivially, when $AB=\bot$  it holds that
$\prev[(C|A)|B]=P(C|A)$, i.e. the prevision of the iterated conditional
``if $B$ then (if $A$ then  $C$)'' coincides with the probability of ``(if $A$ then  $C$)''.
Therefore,  the probability of $B$ does not play a  role in propagating the uncertainty from the premise set $\{B,C|A\}$ to the conclusion $(C|A)|B$. In particular,  if $B=\no{A}$, then   $(C|A)|\no{A}=P(C|A)$. 

This result can be used as a model for some instances of counterfactuals. 
\emph{Counterfactuals} are conditionals in the subjunctive mood, which people usually use 
when they believe that the antecedents are false. For example, the assertion of ``If the glass had fallen  from the table, then it would have broken'' conversationally implies  that the speaker believes that the glass did not fall.  Counterfactuals are important for causal reasoning and for hypothetical thinking in general. There is experimental evidence that people judge the probability of a counterfactual, ``If $A$ were the case, then $C$ would be the case'', as the conditional probability, $P(C | A)$ \cite{over07b,pfeifer2015,pfeiferTulkki17}. Moreover, when presented with causal, e.g., ``If a patient were to take certain  drug, the symptoms would diminish'', or non-causal task material, e.g., ``If the card were to show a square, it would be black'', people judge the negations of the antecedents to be irrelevant to the evaluation of the counterfactuals \cite{pfeifer2015,pfeiferTulkki17}. These negations state the actual facts, e.g., ``The patient does not take the drug'', or ``The side does not show a square'', respectively. This speaks for the psychological plausibility of our basic intuition, which also underlies Stalnaker's extension of the Ramsey test to counterfactuals \cite{edgington14,evans04,stalnaker68}: when we evaluate the counterfactual ``If $A$ were the case, $C$ would be the case'', we hypothetically remove, or set aside, 
our information that A is false from our beliefs and assess $C$ under the assumption that $A$ is true. This matches the psychological data  \cite{pfeifer2015,pfeiferTulkki17}. One starting point of  modeling such situations is given by the aforementioned iterated conditional $(C|A)|B$, with $B\subseteq \no{A}$, where $B$ represents the factual statement which provides evidence that $\no{A}$. 

We remark 
that, contrary to \cite{mcgee89}, in general the iterated conditional $(C|A)|B$, when $A,B,C$ are logically independent, does not  coincide with the conditional event $C|AB$. 
Indeed, by setting $\prev[(C|A)|B]=\mu$ and  $P(C|A)=y$, from Definition \ref{ITER-COND} we obtain
\[
(C|A)|B= (C|A)\wedge B + \mu\no{B}=
\left\{\begin{array}{ll}
1, &\mbox{ if }  ABC \mbox{ is true,}\\
0, &\mbox{ if }  AB\no{C} \mbox{ is true,}\\
y, &\mbox{ if }  \no{A}B \mbox{ is true,}\\
\mu, &\mbox{ if }  \no{B} \mbox{ is true,}\\
\end{array}
\right.
\]
while, assuming  $AB\neq \emptyset$ and  $P(C|AB)=z$, it holds that
\[
C|AB= ABC + z\no{AB}=
\left\{\begin{array}{ll}
1, &\mbox{ if }  ABC \mbox{ is true,}\\
0, &\mbox{ if }  AB\no{C}  \mbox{ is true,}\\
z, &\mbox{ if }  \no{AB} \mbox{ is true;}
\end{array}
\right.
\]
thus: $(C|A)|B\neq C|AB$. Moreover, as 
$(C|A)|B=(AC+y\no{A})|B=AC|B+y\no{A}|B$, by linearity of prevision and  product formula
\begin{equation}
\prev[(C|A)|B]=P(C|AB)P(A|B)+P(C|A)P(\no{A}|B). 
\end{equation}
	Therefore, like in \cite{adams75,Kauf09}, the Import-Export
	Principle is not valid in our approach. Then, as proved in
	\cite{GiSa14}, we avoid the counter-intuitive consequences related to
	Lewis' well-known first triviality result (\cite{lewis76}).
	Moreover, if the Import-Export Principle were added as an axiom to our
	theory, assuming $AB=\emptyset $, $A\neq \emptyset , B\neq \emptyset
	$, we would have on one hand $(C|A)|B =P(C|A)$; on the other hand it
	would be $(C|A)|B= C|AB = C|\emptyset $; thus, we would obtain an
	inconsistency. We also recall that, following de Finetti, objects
	like $C|\emptyset $ are not considered in our approach. Finally, we
	point out that we are able to manage counterfactuals; indeed, in our
	approach the counterfactual $C|A$ when $A$ is believed to be false is
	not $C|\emptyset $, but $(C|A)|\widebar{A}$, which coincides with
	$P(C|A)$.

\section{Conclusions}\label{SEC:CONCLUSION}
We have presented a probabilistic analysis of the conjunction and iteration of conditional events, and of the centering inference for these conjunctions and iterations. In our approach  conjoined conditionals  and iterated conditionals are conditional random quantities defined  in the setting of coherence. By this approach we can overcome some objections made in the past to the conditional probability hypothesis for natural language conditionals, that $P(\textit{if } A \textit{ then } B) = P(B|A)$. This hypothesis is fundamental for the new Bayesian and probabilistic approaches in the psychology of reasoning and has been confirmed in many papers (\cite{over2016,OverCruz17,pfeifer,pfeifer10a,pfeifer10b,pfeifertulkkiCogsci17,pfeiferyama17}).  This identity is also central to our analysis of both indicative and counterfactual conditionals as conditional events.

We have proved the p-validity of one-premise and two-premise centering when basic events are replaced by conditional events. We have determined the lower and upper bounds for the conclusion of two-premise centering; we have also studied the classical case and have obtained the same lower and upper bounds. We have proved the p-validity of an analogue of the classical biconditional introduction rule for conditional events (biconditional AND rule). We have verified the p-validity of one-premise and two-premise biconditional centering, and have given the lower and upper bounds for the conclusion of two-premise biconditional centering. We have investigated reversed inferences, by determining the lower and upper bounds for the premises of the biconditional AND rule. We have briefly indicated how to apply our results to the study of selected counterfactuals.

 It is often argued that there are deep differences between indicative and counterfactual conditionals. For example, the indicative conditional, ``If Oswald did not kill Kennedy then someone else did'', is not equivalent to the counterfactual conditional, ``If Oswald had not killed Kennedy then someone else would have'' (\cite{edgington14}). We will explore in future work the more detailed similarities and differences between these two forms of the conditional.
\section*{Acknowledgement}
We are grateful to  anonymous referees for helpful comments and suggestions. 
We thank \emph{Deutsche Forschungsgemeinschaft} (DFG),  \emph{Fondation Maison des Sciences de l'Homme} (FMSH), and \emph{Villa Vigoni} for supporting joint meetings at Villa Vigoni where parts of this work originated (Project: ``Human Rationality: Probabilistic Points of View''). Niki Pfeifer is supported by his DFG project PF~740/2-2 (within the SPP1516 ``New Frameworks of Rationality'').  Giuseppe Sanfilippo has been partially supported by the INdAM--GNAMPA Project (2016 Grant U 2016/000391). 
% check the number!!
\appendix
\section{Some technical aspects on coherence}
In this Appendix, which expands  Section~\ref{sect:2.2}, we 
illustrate some technical aspects which concern coherence of probability and prevision assessments on conditional events and conditional random quantities.
\label{Appendix}
\subsection{Coherence Checking}
Let be given a family $\mathcal{F}_n = \{E_1|H_1, \ldots, E_n|H_n\}$ and a probability assessment $\mathcal{P}_n =(p_1, \ldots, p_n)$ on $\mathcal{F}_n$,
where $p_i = P(E_i|H_i) \, ,\;\; i = 1, \ldots, n$.
We set $J_n=\{1,2,\ldots,n\}$.
We recall that the constituents $C_0,C_1,\ldots, C_m$ generated by $\mathcal{F}_n$, where  $C_0=\no{\H}_n=\no{H_1}\no{H_2}\cdots \no{H_n}$,  are obtained by expanding the expression  $\bigwedge_{i=1}^n(E_iH_i \vee \no{E}_iH_i
\vee \no{H}_i)$.
For  each $h\in J_m$, with the constituent $C_h$  we associate a point
$Q_h = (q_{h1}, \ldots, q_{hn})$, where for each $j\in J_n$, $q_{hj} = 1$, or 0, or $p_j$, according to whether $C_h \subseteq E_jH_j$, or $C_h \subseteq \no{E}_jH_j$, or $C_h \subseteq \no{H}_j$.
Denoting by $\mathcal{I}$ the convex hull of  $Q_1, \ldots, Q_m$,
 by a suitable alternative theorem (\cite[Theorem~2.9]{Gale60}), the condition $\mathcal{P}_n \in \mathcal{I}$
is equivalent to  the condition $\min  \mathcal{D}_{\mathcal{H}_n} \leq 0 \leq \max
\mathcal{D}_{\mathcal{H}_n}$ given in Definition~\ref{COER-BET}  (see, e.g., \cite{Gili96,gilio13ins}). Moreover, the condition $\mathcal{P}_n \in \mathcal{I}$ amounts to the  solvability of the
following system ($\Sigma$) in the unknowns $\lambda_1, \ldots,
\lambda_m$
\[
(\Sigma): \hspace{1 cm}
\sum_{h=1}^m q_{hj} \lambda_h = p_j \; , \; \; j\in J_n \,
;\;\; \sum_{h=1}^m \lambda_h = 1 \; ;\; \; \lambda_h \geq 0 \, ,
\; h \in J_m\,.
\]
We say that system $(\Sigma)$ is  associated with the pair $(\mathcal{F}_n,\mathcal{P}_n)$.  Hence, the following result
 provides a characterization  of  the notion of coherence given in Definition~\ref{COER-BET} (\cite[Theorem~4.4]{Gili90}, see also \cite{Gili92,GiSa11a, gilio13ins})

\begin{theorem}\label{CNES}	Let $\K$ be  an arbitrary family of  conditional events  and let $P$ be a probability function  defined on $\K$. The  function $P$ is coherent if and only if, for
every finite subfamily $\F_n=\{E_1|H_1, \ldots, E_n|H_n\}$  of $\K$, denoting by  $\mathcal{P}_n$ the vector $(p_1, \ldots, p_n )$, where $p_j=P(E_j|H_j)$, $j=1,2,\ldots,n$, the 
		 system $(\Sigma)$ associated with the pair
		$(\F_n,\mathcal{P}_n)$ is solvable.
		  \end{theorem}
We recall now some results on the coherence checking of a probability assessment on a finite family of conditional events. 
Given a probability  assessment  $\mathcal{P}_n=(p_1, \ldots, p_n )$ on a finite family of conditional events $\mathcal{F}_n=\{E_1|H_1, \ldots, E_n|H_n\}$, let $S$ be the set of solutions $\Lambda = (\lambda_1, \ldots,
\lambda_m)$ of the system $(\Sigma)$. Then, assuming $S \neq \emptyset$, we define
\[\begin{array}{l}
\Phi_j(\Lambda) = \Phi_j(\lambda_1, \ldots, \lambda_m) = \sum_{r :
	C_r \subseteq H_j} \lambda_r \; , \; \; \; j \in J_n \,;\; \Lambda \in S \,;
\\
M_j  =  \max_{\Lambda \in S } \; \Phi_j(\Lambda) \; , \; \; \; j\in J_n\,;
\;\;\; I_0  =  \{ j \, : \, M_j=0 \} \,.
\end{array}\]
Of course, if   $S \neq \emptyset$, then $S$ is a closed bounded set and the  maximum  $M_j$ of the linear function $\Phi_j(\Lambda)=\sum_{r :
	C_r \subseteq H_j} \lambda_r$  there exists   for every $j\in J_n$. 
We observe that, assuming $\P_n$ coherent, each solution $\Lambda=(\lambda_1, \ldots,
\lambda_m)$ of system $(\Sigma)$ is a coherent extension of the assessment $\mathcal{P}_n$ on $\mathcal{F}_n$ to the family $\{C_1|\H_n,C_2|\H_n,\, \ldots,\,
C_m|\H_n\}$. Then, by the additivity property, the quantity $\Phi_j(\Lambda)$ is the conditional probability $P(H_j|\H_n)$ and the quantity $M_j$ is the upper probability $P^*(H_j|\H_n)$ over all the solutions $\Lambda$ of system $(\Sigma)$.
Of course, $j \in I_0$ if and only if $P^*(H_j|\H_n)=0$. Notice that $I_0$  is a strict subset of $J_n$. We denote by $(\mathcal{F}_0, \mathcal{P}_0)$ the pair associated with $I_0$.
Given the pair $(\mathcal{F}_n,\mathcal{P}_n)$ and a (nonempty) strict subset $J$ of $ J_n$, we denote by $(\mathcal{F}_J, \mathcal{P}_J)$ the pair associated with
$J$ and by $(\Sigma_J)$ the corresponding system.
We observe that $(\Sigma_J)$ is solvable if and only if $\mathcal{P}_J  \in \mathcal{I}_J$,
where $\mathcal{I}_J$ is the convex hull associated with the pair $( \mathcal{F}_J,
\mathcal{P}_J)$. Then, we have  (\cite[Theorem~3.2]{Gili93}; see also \cite{2003BGS-IJUFKS,Gili95})
\begin{theorem}\label{GILIO-93}{\rm
		Given a probability assessment $\mathcal{P}_n$ on the family $\mathcal{F}_n$, if
		the system $(\Sigma)$ associated with $(\mathcal{F}_n,\mathcal{P}_n)$ is solvable, then for every $J\subset J_n$, such that $J\setminus I_0\neq \emptyset$, the system $(\Sigma_J)$ associated with $(\mathcal{F}_J,\mathcal{P}_J)$ is solvable too.}
\end{theorem}
The previous result says that the condition $\P_n \in \I$ implies $\P_J \in \I_J$ when $J\setminus I_0\neq \emptyset$. We observe that, if $\P_n \in \I$, then for every nonempty subset $J$ of $J_n\setminus I_0$ it holds that $J\setminus I_0=J \neq \emptyset$; hence, by Theorem \ref{CNES}, the subassessment $\P_{J_n\setminus I_0}$
on the subfamily $\F_{J_n\setminus I_0}$ is coherent. In particular, when $I_0$ is empty, coherence of $\P_n$ amounts to solvability of system $(\Sigma)$, that is to condition $\P_n \in \I$. When $I_0$ is not empty, coherence of $\P_n$ amounts to the validity of both conditions $\P_n \in \I$ and $\P_0$ coherent, as shown  below (\cite[Theorem~3.3]{Gili93}).
\begin{theorem}\label{COER-P0}{\rm The assessment $\mathcal{P}_n$ on $\mathcal{F}_n$ is coherent if and only if the following conditions are satisfied:
		(i) $\mathcal{P}_n \in \mathcal{I}$; (ii) if $I_0 \neq \emptyset$, then $\mathcal{P}_0$ is coherent.
}\end{theorem}
\subsection{Coherent conditional prevision assessments}
We denote by $X$ a random quantity, that is (following de Finetti, see also \cite{lad96}) an 
uncertain real quantity,  which has a well determined but unknown value. We recall that in the axiomatic approach to probability, usually $X$ is defined as a random variable.
We assume that $X$ is a finite random quantity, that is $X$ has a finite set of possible values. 
In particular,  (the indicator of) any given event $A$ is a two-valued random quantity, with $A\in\{0,1\}$. 
Given an event $H \neq \bot$ and a finite random quantity  $X$,  let  $\{x_1, x_2, \ldots, x_r\}$ be
the set of possible values of $X$ restricted to $H$, which means that 
if $H$ is true then $X\in \{x_1, x_2, \ldots, x_r\}$.  As an example, given an event $H\neq \bot$ and a finite random quantity $X$ with a set of  possible values  $\{x_1, x_2, \ldots, x_n\}$.
Let be $A_i=(X=x_i)$, $i=1,2,\ldots,n$. Assume that $A_iH\neq \bot$ for $i=i_1,i_2,\ldots,i_r$ and with  $A_iH= \bot$ otherwise. Then, the set of possible values of $X$ restricted to $H$ is $\{x_{i_1}, x_{i_2}, \ldots, x_{i_r}\}$. Notice that the set $\{x_{i_1}, x_{i_2}, \ldots, x_{i_r}\}$ is nonempty because $\bigvee_{i=1}^n A_iH=(\bigvee_{i=1}^n A_i)H=\Omega H=H\neq \bot$. Indeed, if it were $A_iH=\bot$ for $i=1,\ldots,n$, then it would follows that $H=\bigvee_{i=1}^n A_iH=\bot$ (which is a contradiction).
 
 Agreeing to the betting metaphor, by assessing the prevision of $``X$ {\em conditional on} $H$'' (also named $``X$ {\em given} $H$''), $\pr(X|H)$,  as the amount $\mu$, then for any given  real number $s$ you are willing to pay an amount $\mu s$ and to receive  $Xs$, or $\mu s$, according  to whether $H$ is true, or  false (bet called off), respectively. Then, the random gain associated with the assessment $\pr(X|H)=\mu$ is 
 $G = sH(X-\mu)$.
 We remark that, differently from the notion of  expected value,  in the subjective approach of de Finetti
 the prevision of a random quantity is a primitive notion and its value can be assessed in a direct way. 
  In particular, when $X$ is (the indicator of) an event $A$, then $\prev(X|H)=P(A|H)$. We recall the notion of coherence for conditional prevision assessments (\cite{BiGS08,BiGS12,GiSa14,PeVa17}). 
Given a  function $\pr$ defined on an arbitrary family $\K$ of finite
conditional random quantities, consider a finite subfamily $\F_n = \{X_i|H_i, \, i
\in J_n\} \subseteq \K$  and the vector
$\M_n=(\mu_i, \, i \in J_n)$, where $\mu_i = \pr(X_i|H_i)$, $i\in J_n$.
With the pair $(\F_n,\M_n)$ we associate the random gain $G =
\sum_{i \in J_n}s_iH_i(X_i - \mu_i)$, where $s_1,s_2,\ldots,s_n$ are arbitrary real numbers; moreover, as made for the conditional events, we 
 denote by $\mathcal{D}$ the set of values of $G$ and by
$\mathcal{D}_{\mathcal{H}_n}$, where  $\H_n = H_1 \vee \cdots \vee H_n$,  the set of values of $G$ restricted to $\H_n$. Then, using the {\em betting scheme} of de Finetti, we
have
\begin{definition}\label{COER-RQ}{\rm
		The function $\pr$ defined on $\K$ is coherent if and only if, $\forall n
		\geq 1$,  $\forall \, \F_n \subseteq \K,\, \forall \, s_1, \ldots,
		s_n \in \mathbb{R}$, it holds that: $min \; \mathcal{D}_{\mathcal{H}_n} \; \leq 0 \leq max \;
		\mathcal{D}_{\mathcal{H}_n}$. }
\end{definition}
Given a family $\F_n = \{X_1|H_1,\ldots,X_n|H_n\}$, for each $i \in J_n$ we denote by $\{x_{i1}, \ldots,x_{ir_i}\}$ the set of possible values for the restriction of $X_i$ to $H_i$; then, for each $i \in J_n$ and $j = 1, \ldots, r_i$, we set $A_{ij} = (X_i = x_{ij})$. Of course, for each $i \in J_n$, the family $\{\no{H}_i, A_{ij}H_i \,,\; j = 1, \ldots, r_i\}$ is a partition of the sure event $\Omega$, with  $A_{ij}H_i=A_{ij}$, $\bigvee_{j=1}^{r_i}A_{ij}=H_i$. Then,
the constituents generated by the family $\F_n$ are (the
elements of the partition of $\Omega$) obtained by expanding the
expression $\bigwedge_{i \in J_n}(A_{i1} \vee \cdots \vee A_{ir_i} \vee
\no{H}_i)$. We set $C_0 = \no{H}_1 \cdots \no{H}_n$ ($C_0$ may be equal to $\bot$);
moreover, we denote by $C_1, \ldots, C_m$ the constituents
contained in $\H_n = H_1 \vee \cdots \vee H_n$. Hence
$\bigwedge_{i \in J_n}(A_{i1} \vee \cdots \vee A_{ir_i} \vee
\no{H}_i) = \bigvee_{h = 0}^m C_h$.
With each $C_h,\, h \in J_ m$, we associate a vector
$Q_h=(q_{h1},\ldots,q_{hn})$, where $q_{hi}=x_{ij}$ if $C_h \subseteq
A_{ij},\, j=1,\ldots,r_i$, while $q_{hi}=\mu_i$ if $C_h \subseteq \no{H}_i$;
the vector associated with $C_0$ is $Q_0=\M_n = (\mu_1,\ldots,\mu_n)$. 
Denoting by $\I$ the convex hull of $Q_1, \ldots, Q_m$, the condition  $\M_n\in \I$ amounts to the existence of a vector $(\lambda_1,\ldots,\lambda_m)$ such that:
$ \sum_{h \in J_m} \lambda_h Q_h = \M_n \,,\; \sum_{h \in J_m} \lambda_h
= 1 \,,\; \lambda_h \geq 0 \,,\; \forall \, h$; in other words, $\M_n\in \I$ is equivalent to the solvability of the following system,  associated with  $(\F_n,\M_n)$,
\begin{equation}\label{SYST-SIGMA}
\begin{array}{l}
\sum_{h \in J_m} \lambda_h q_{hi} =
\mu_i \,,\; i \in J_n \,; \; \sum_{h \in J_m} \lambda_h = 1 \,;\;
\lambda_h \geq 0 \,,\;  \, h\in J_m \,.
\end{array}
\end{equation}
Given the assessment $\M_n =(\mu_1,\ldots,\mu_n)$ on  $\F_n =
\{X_1|H_1,\ldots,X_n|H_n\}$, let $S$ be the set of solutions $\Lambda = (\lambda_1, \ldots,\lambda_m)$ of system   (\ref{SYST-SIGMA}).   Then, assuming  the system (\ref{SYST-SIGMA})  solvable, that is  $S \neq \emptyset$, we define:
\begin{equation}\label{EQ:I0}
I_0 = \{i : \max_{\Lambda \in S} \, \sum_{h:C_h\subseteq H_i}\lambda_h= 0\},\;  \F_0 = \{X_i|H_i \,, i \in I_0\},\;  \M_0 = (\mu_i ,\, i \in I_0)\,.
\end{equation}
Then, the following theorem can be proved  (\cite[Theorem~3]{BiGS08}):
\begin{theorem}\label{CNES-PREV-I_0-INT}{\rm [{\em Operative characterization of coherence}]
		A conditional prevision assessment ${\M_n} = (\mu_1,\ldots,\mu_n)$ on
		the family $\F_n = \{X_1|H_1,\ldots,X_n|H_n\}$ is coherent if
		and only if the following conditions are satisfied: \\
		(i) the system  (\ref{SYST-SIGMA}) is solvable;\\ (ii) if $I_0 \neq \emptyset$, then $\M_0$ on $\F_0$ is coherent. }
\end{theorem}
%Given a subset $S'\subseteq S$, let $I_0'$ be the set defined as in (\ref{EQ:I0}), where $S$ is replaced by $S'$ and denote by $(\F_0',\M_0')$ the pair associated with $\I_0'$. Of course, $\I_0\subseteq \I_0'$. Then, we have the following result (\cite{SC03}) 
%
%\begin{theorem}\label{CNESI_0PRIME}{\rm
%		A conditional prevision assessment ${\M_n} = (\mu_1,\ldots,\mu_n)$ on
%		the family $\F_n = \{X_1|H_1,\ldots,X_n|H_n\}$ is coherent if
%		and only if the following conditions are satisfied: \\
%		(i) the system $(\Sigma)$ defined in (\ref{SYST-SIGMA}) is solvable;\\
%		(ii) if $I_0'\neq \emptyset $, then $\M_0'$ on $\F_0'$ is coherent. }
%\end{theorem}
\subsection{Conditional previsions as previsions of conditional random quantities}
		\label{SEC:EXTENDED}
We recall that usually in the literature a conditional random quantity $X|H$ is understood as the restriction of $X$ to $H$, with $X|H$ undefined when $H$ is false. 
By the betting scheme,   if you assess $\pr(X|H)=\mu$, the random quantity that you receive
by paying $\mu$  is   $XH + \mu \no{H}$. From coherence, it holds that  $\prev(XH + \mu \no{H})=\mu$; indeed, if you would assess $\prev(XH + \mu \no{H})=\mu^*\neq \mu$, then the random gain associated with  the assessment $(\mu,\mu^*)$  on  $\{X|H,XH + \mu \no{H}\}$ would be 
\[
G=s_1H(X-\mu)+s_2(XH + \mu \no{H}-\mu^*)\,.
\]
Then, by choosing $s_1=1$ and $s_2=-1$, the random gain $G$ would be equal to the nonzero constant:  $\mu^*-\mu$ (a Dutch book). In what follows, by the symbol $X|H$ we denote  the random quantity $XH+\mu \no{H}$, where  $\mu=\pr(X|H)$. This random quantity, which  extends the restriction of $X$ to $H$,  coincides with  $X$ when $H$ is true and is equal to $\mu$ when $H$ is false  (\cite{GiSa14}; see also \cite{GiSa13c,GiSa13a,lad96}). As shown before the conditional prevision $\pr(X|H)$ is the prevision of the conditional random quantity $X|H$.
In this way, based on the betting scheme  $X|H$ is the  amount that you receive in a bet on $X$ conditional on $H$, if you agree to pay $\pr(X|H)$.  
Moreover,  denoting by $\{x_1, x_2, \ldots, x_r\}$ the set of possible values of $X$ when $H$ is true, and defining $A_i=(X=x_i)$, $i=1,2,\ldots,r$, the family $\{A_1H,\ldots,A_rH,\no{H}\}$ is a partition of the sure event $\Omega$ and we have
\[
X|H = XH + \mu \no{H} = x_1A_1H + \cdots + x_rA_rH + \mu \no{H} \in\{x_1,x_2,\ldots,x_r,\mu\}\,.
\]
In particular, when $X$ is (the indicator of) an event $A$, the prevision of $X|H$ is the probability of the conditional event $A|H$ and, if you assess $P(A|H) = p$, then for the indicator of $A|H$ we have $A|H = AH + p\no{H} \in \{1,0,p\}$. We observe that the choice of $p$ as the value of $A|H$ when $H$ is false has been also considered in some previous works  (\cite{coletti02,Gili90,Jeff91,lad96,mcgee89,StJe94,vanF76}).
\bibliographystyle{plain} 

%\bibliography{IJARCenteringbiblio}

\end{document}